\documentclass[a4paper,oneside,11pt]{article}
\usepackage[utf8]{inputenc}
\usepackage{lmodern,stackrel}
\usepackage[T1]{fontenc}
\usepackage{verbatim}
\usepackage{textcomp}
\usepackage[english]{babel}
\usepackage[pdftex]{hyperref}
\usepackage[pdftex]{color,graphicx}
\usepackage{amssymb}
\usepackage{amsmath}
\usepackage{upgreek}
\usepackage{subcaption}
\usepackage{epstopdf}
\usepackage{dsfont}
\usepackage{tikz}

\renewcommand{\leq}{\leqslant}
\renewcommand{\geq}{\geqslant}

\usepackage{amsmath,amsfonts,amssymb,amsthm,mathrsfs,stackrel}

\def\build#1_#2^#3{\mathrel{
\mathop{\kern 0pt#1}\limits_{#2}^{#3}}}

\theoremstyle{plain}
\newtheorem{theorem}{Theorem}

\newtheorem{corollary}{Corollary}
\newtheorem{proposition}[corollary]{Proposition}

\newtheorem{lemma}{Lemma}

\theoremstyle{definition}

\theoremstyle{claim}

\theoremstyle{remark}
\newtheorem*{remark}{Remark}

\begin{document}
\global\long\def\dist#1#2#3{\operatorname{dist}_{{#1}}\left({#2},{#3}\right)}%

\title{A note on inhomogeneous percolation on ladder graphs}
\author{Bernardo N. B. de Lima\footnote{Departamento de Matem{\'a}tica, Universidade Federal de Minas Gerais, Av. Ant\^onio
Carlos 6627 C.P. 702 CEP 30123-970 Belo Horizonte-MG, Brazil} \and Humberto C. Sanna$^*$}

\date{}
\maketitle

\begin{abstract}
Let $\mathbb{G}=\left(\mathbb{V},\mathbb{E}\right)$ be the graph
obtained by taking the cartesian product of an infinite and connected
graph $G=(V,E)$ and the set of integers $\mathbb{Z}$. We choose
a collection $\mathcal{C}$ of finite connected subgraphs of $G$
and consider a model of Bernoulli bond percolation on $\mathbb{G}$
which assigns probability $q$ of being open to each edge whose projection
onto $G$ lies in some subgraph of $\mathcal{C}$ and probability
$p$ to every other edge. We show that the critical percolation threshold
$p_{c}\left(q\right)$ is a continuous function in $\left(0,1\right)$,
provided that the graphs in $\mathcal{C}$ are ``well-spaced'' in
$G$ and their vertex sets have uniformly bounded cardinality. This generalizes a recent result due to Szab\'o and Valesin. 
\end{abstract}


\section{Introduction}

In this note we address a particular case of the following problem:
let $\mathbb{G}=\left(\mathbb{V},\mathbb{E}\right)$ be an infinite,
connected graph, and $\mathbb{E}',\mathbb{E}''$ a decomposition of
the edge set $\mathbb{E}$. Consider the Bernoulli percolation model
in which the edges of $\mathbb{E}'$ are open with probability $p$
and the edges of $\mathbb{E}''$, regarded as the set of inhomogeneities,
are open with probability $q$. If we define the quantity $p_{c}(q)$
as the supremum of the values of $p$ for which percolation with parameters
$p,q$ does not occur, what can we say about the behavior of the function
$q\mapsto p_{c}(q)$?

Perhaps one of the earliest works concerning this type of problem
is due to Kesten, presented in \cite{Ke}. Considering the
square lattice $\mathbb{L}^{2}=\left(\mathbb{Z}^{2},\mathbb{E}\right)$
and choosing $\mathbb{E}''$ and $\mathbb{E}'$ to be respectively
the sets of vertical and horizontal edges, he proves that $p_{c}(q)=1-q$.
Later on, in \cite{Z}, Zhang also considers the square lattice,
but with the edge set $\mathbb{E}''$ being only the vertical edges
within the $y$-axis and $\mathbb{E}'=\mathbb{E}\setminus\mathbb{E}''$.
He proves that for any $q<1$ there is no percolation at $p=1/2$,
which implies that $p_{c}(q)$ is constant in the interval $[0,1)$.
In the context of long-range percolation, the authors in \cite{LRV}
consider an oriented, $d$-regular, rooted tree $\mathbb{T}_{d,k}$,
where besides the usual set of ``short bonds'' $\mathbb{E}'$, there
is a set $\mathbb{E}''$ of ``long edges'' of length $k\in\mathbb{N}$,
pointing from each vertex $x$ to its $d^{k}$ descendants at distance
$k$. They show that $q\mapsto p_{c}(q)$ is continuous and strictly
decreasing in the region where it is positive. This conclusion is
also achieved in \cite{CLS}, where the authors consider the slab
of thickness $k$ induced by the vertex set $\mathbb{Z}^{2}\times\{0,\ldots,k\}$,
with $\mathbb{E}'$ and $\mathbb{E}''$ being respectively the sets
of edges parallel and perpendicular to the $xy$-plane.

Another work that we mention is that of Iliev, Janse van Rensburg
and Madras, \cite{IRM}. In the context of bond percolation in $\mathbb{Z}^{d}$,
they define $\mathbb{E}''$ to be the set of edges within the subspace
$\mathbb{Z}^{s}\times\{0\}^{d-s}$, $2\leq s<d$, and study the behavior
of the quantity $q_{c}(p)$, defined analogously to $p_{c}(q)$. Among
other standard results, the authors prove that $q_{c}(p)$ is strictly
decreasing in the interval $[0,p_{c}]$, where $p_{c}$ is the percolation
threshold in the homogeneous case.

More recently, in \cite{SV}, Szab\'o and Valesin consider the same 
framework for $\mathbb{G}$, $\mathbb{E}'$ and $\mathbb{E}''$
and prove that, under this setting, $p_{c}(q)$ is continuous in the
interval $(0,1)$. In their model, the graph $\mathbb{G}$ is obtained
by taking the cartesian product of an infinite and connected graph $G=(V,E)$
and the set of integers $\mathbb{Z}$. The set of inhomogeneities
$\mathbb{E}''$ is constructed by selecting a finite number of infinite
``columns'' and ``ladders'' and considering all the edges within it,
and $\mathbb{E}'=\mathbb{E}\setminus\mathbb{E}''$.

It is in the spirit of \cite{SV} that we approach the aforementioned
problem. More specifically, we extend their result in the sense that
the continuity of $p_{c}(q)$ also holds when we set parameter $q$
on infinitely many ``ladders'' and ``columns'', as long as they are
``well spaced''.

\subsection{\label{subsec:ladder-graphs-result}Inhomogeneous percolation on
ladder graphs: definitions and result}

Let $G=(V,E)$ be an infinite, connected and bounded degree graph
with vertex set $V$ and edge set $E$. Starting from $G$, we define
the graph $\mathbb{G}=(\mathbb{V},\mathbb{E})$, where $\mathbb{V}:= V\times\mathbb{Z}$
and 
\[
\mathbb{E}:=\bigl\{\left\langle \left(u,n\right),\left(v,n\right)\right\rangle ;\left\langle u,v\right\rangle \in E,n\in\mathbb{Z}\bigr\}\cup\bigl\{\left\langle \left(w,n\right),\left(w,n+1\right)\right\rangle ; w\in V,n\in\mathbb{Z}\bigr\}.
\]

Consider the Bernoulli percolation process on $\mathbb{G}$ described
as follows. Every edge of $\mathbb{E}$ can be \emph{open} or \emph{closed},
states which shall be represented by $1$ and $0$, respectively.
Hence, a typical percolation configuration is an element of $\Omega=\{0,1\}^{\mathbb{E}}$.
As usual, the underlying $\sigma$-algebra is the one generated by
the finite-dimensional cylinder sets of $\Omega$. For the probability
measure of the process, we shall define it based on the rule specified
below:

Fix a family of subgraphs $\left\{ G^{(r)}=\left(U^{(r)},E^{(r)}\right)\right\} _{r\in\mathbb{N}}$
of $G$, such that:
\begin{itemize}
\item $G^{(r)}$ is finite and connected for every $r\in\mathbb{N}$;
\item $\dist G{U^{(i)}}{U^{(j)}}\geq3,\forall i\neq j$ (where $\dist G{.}{.}$ denotes the graph distance).
\end{itemize}
For each $r\in\mathbb{N}$, let
\begin{equation}
\begin{split}\mathbb{E}^{\text{in},(r)} & :=\left\{ \left\langle \left(u,n\right),\left(v,n\right)\right\rangle ;\langle u,v\rangle\in E^{(r)},n\in\mathbb{Z}\right\} \\
 & \quad\cup\left\{ \left\langle \left(w,n\right),\left(w,n+1\right)\right\rangle ; w\in U^{(r)},n\in\mathbb{Z}\right\} 
\end{split}
\label{eq:internal-cylinder-infinite}
\end{equation}

Given $p\in\left[0,1\right]$ and $q\in(0,1)$, declare each edge
of $\mathbb{E}^{\text{in},(r)}$ open with probability $q$, independently
of all other edges, for every $r\in\mathbb{N}$. Likewise, declare
each edge of $\mathbb{E}\setminus\left(\cup_{r\in\mathbb{N}}\mathbb{E}^{\text{in},(r)}\right)$
open with probability $p$, also independently of any other edge.
Let $\mathbb{P}_{q,p}$ be the law of the open edges for the process
just described.

Having established our model, we turn our attention to state the main
result of this section. First, a few definitions are required.

An \emph{open path} in $\mathbb{G}$ is a set of distinct vertices
$\left(v_{0},n_{0}\right),\left(v_{1},n_{1}\right),\ldots,\left(v_{m},n_{m}\right)$
such that for every $i=0,\ldots,m-1$, $\left\{ \left(v_{i},n_{i}\right),\left(v_{i+1},n_{i+1}\right)\right\} \in\mathbb{E}$
and is open. Given $\omega\in\Omega$ and $\left(v_{0},n_{0}\right),(v,n)\in\mathbb{E}$,
we say that $(v,n)$ can be reached from $\left(v_{0},n_{0}\right)$
in the configuration $\omega$ either if the two vertices are equal
or if there is an open path from $\left(v_{0},n_{0}\right)$ to $(v,n)$.
Denote this event by $\left(v_{0},n_{0}\right)\leftrightarrow(v,n)$; we also use the notation $(v_0,n_0)\stackrel{S}{\leftrightarrow}(v,n)$ to denote the event where there exists an open path connecting $(v_{0},n_{0})$ and $(v,n)$ with all vertices belonging to the set $S$. The \emph{cluster} $C_{(v,n)}$ of $(v,n)$ in the configuration $\omega$
is the set of vertices that can be reached from $(v,n)$. That is,
\[
C_{(v,n)}:=\left\{ \left(u,m\right)\in\mathbb{V}; (v,n)\leftrightarrow(u,m)\right\} .
\]
In particular, we denote $C_{v}=C_{(v,0)}$. If $\left|C_{(v,n)}\right|=\infty$,
we say that the vertex $(v,n)$ \emph{percolates} and write $\left\{ (v,n)\leftrightarrow\infty\right\} $
for the set of such realizations.

Now, fix $v\in V$ and note that whether or not $\mathbb{P}_{q,p}\left((v,0)\leftrightarrow\infty\right)>0$
depends on the values of the parameters $p$ and $q$. With this in
mind, we define the \emph{critical curve} of our model as a function
of $q$, namely
\[
p_{c}(q):=\sup\left\{ p\in[0,1];\mathbb{P}_{q,p}\left((v,0)\leftrightarrow\infty\right)=0\right\} .
\]
One should observe that although the probability $\mathbb{P}_{q,p}\left((v,0)\leftrightarrow\infty\right)$
may vary from vertex to vertex, the value of $p_{c}(q)$ does not
depend on the choice of $v\in V$, since $\mathbb{G}$ is connected.

What we shall prove in the next section is, in some sense, a generalization
of Theorem 1 in \cite{SV}. It states that the continuity of $p_{c}(q)$
still holds, provided that the cardinality of the sets $U^{(r)}$
are uniformly bounded.
\begin{theorem}
\label{thm:continuity-pcrit} If $\sup_{r\in\mathbb{N}}\left|U^{(r)}\right|<\infty$ and $\dist G{U^{(i)}}{U^{(j)}}\geq3,\forall i\neq j$,
then $q\mapsto p_{c}(q)$ is continuous in $(0,1)$.
\end{theorem}
\begin{remark}
Just as we have based our non-oriented percolation model upon the
one of Szab\'o and Valesin, we can generalize the oriented model also
present in \cite{SV} in an analogous manner. By the same reasoning
we shall present in the sequel, the continuity of the critical parameter
for this new model also holds.
\end{remark}

\section{Proof of Theorem \ref{thm:continuity-pcrit}}

Theorem \ref{thm:continuity-pcrit} is a consequence of the following
proposition:
\begin{proposition}
\label{prop:worse-q-better-p}Fix $p,q\in(0,1)$ and $\lambda=\min(p,1-p)$. If $\sup_{r\in\mathbb{N}}\left|U^{(r)}\right|<\infty$ and $\dist G{U^{(i)}}{U^{(j)}}\geq3,\forall i\neq j$, for all $\varepsilon\in(0,\lambda)$, there exists $\eta=\eta(q,p,\varepsilon)>0$
such that if $\delta\in(0,\eta)$ then
\[
\mathbb{P}_{q+\delta,p-\varepsilon}\left((v,0)\leftrightarrow\infty\right)\leq\mathbb{P}_{q-\delta,p+\varepsilon}\left((v,0)\leftrightarrow\infty\right)
\]
for every $v\in V\setminus\left(\cup_{r\in\mathbb{N}}U^{(r)}\right)$.
\end{proposition}
\begin{proof}[Proof of Theorem \ref{thm:continuity-pcrit}]
Since $q\mapsto p_{c}(q)$ is non-increasing, any discontinuity,
if exists, must be a jump. Suppose $p_{c}$ is discontinuous at some
point $q_{0}\in(0,1)$, let $a=\lim_{q\downarrow q_{0}}p_{c}(q)$
and $b=\lim_{q\uparrow q_{0}}p_{c}(q)$. Then, for any $p\in(a,b)$,
we can find an $\varepsilon>0$ such that for every $\delta>0$ we
have
\[
\mathbb{P}_{q_{0}-\delta,p+\varepsilon}\left((v,0)\leftrightarrow\infty\right)=0<\mathbb{P}_{q_{0}+\delta,p-\varepsilon}\left((v,0)\leftrightarrow\infty\right)
\]
for every $v\in V$, a contradiction according to Proposition \ref{prop:worse-q-better-p}.
\end{proof}
The proof of Proposition \ref{prop:worse-q-better-p} is based
on the construction of a coupling which allows us to understand how
a small change in the parameters of the model affects the percolation
behavior. This construction is done in several steps. First, we split
our edge set $\mathbb{E}$ in an appropriate disjoint family of subsets.
Second, we define coupling measures on each of these sets in such
a way that the increase of one parameter compensates an eventual decrease
of the other in the sense of preserving the connections between boundary
vertices of some ``well chosen sets'', which will play an important
role when we consider percolation on the graph $\mathbb{G}$ as a
whole. Third, we verify that we can set the same parameters for each
coupling provided that we can limit the size of the sets in which
the inhomogeneities are introduced. Finally, we merge these couplings
altogether by considering the product measure of each one. Most of
these ideas are the same as in \cite{LRV} and \cite{SV}. To
put it rigorously, we begin with some definitions.

For $r\in\mathbb{N}$, $n\in\mathbb{Z}$, let $L_{r}:=\left|U^{(r)}\right|$
and
\begin{align*}
\mathbb{V}_{n}^{(r)} & :=\left\{ (v,m)\in\mathbb{V}; \dist Gv{U^{(r)}}\leq1,(2L_{r}+2)n\leq m\leq(2L_{r}+2)(n+1)\right\} ;\\
\mathbb{E}_{n}^{(r)} & :=\left\{ e\in\mathbb{E}; \text{\ensuremath{e} has both endvertices in \ensuremath{\mathbb{V}_{n}^{(r)}}}\right\} \\
 & \qquad\setminus\left\{ e\in\mathbb{E};  e=\langle(u,(2L_{r}+2)(n+1)),(v,(2L_{r}+2)(n+1))\rangle,\langle u,v\rangle\in\mathbb{E}\right\} ;\\
\mathbb{E}^{(r)} & :=\cup_{n\in\mathbb{Z}}\mathbb{E}_{n}^{(r)}.
\end{align*}
Note that
\begin{itemize}
\item $G$ has bounded degree and $\left|U^{(r)}\right|<\infty$ implies
$\left(\mathbb{V}_{n}^{(r)},\mathbb{E}_{n}^{(r)}\right)$ is finite;
\item $\mathbb{E}_{n}^{(r)}\cap\mathbb{E}_{n'}^{(r)}=\emptyset, \forall n\neq n'$;
\item For any $n,n'\in\mathbb{Z},\ \mathbb{E}_{n}^{(r)}\cap\mathbb{E}_{n'}^{(r')}=\emptyset,\forall r\neq r'$. This is true since we are assuming $\dist G{U^{(r)}}{U^{(r')}}\geq3$,
which implies $\dist{\mathbb{G}}{\mathbb{V}_{n}^{(r)}}{\mathbb{V}_{n'}^{(r')}}\geq1$.
\end{itemize}
Next, recall the definition of $\mathbb{E}^{\text{in},(r)}$ in (\ref{eq:internal-cylinder-infinite})
and define
\begin{align*}
\mathbb{E}_{n}^{\partial,(r)} & :=\mathbb{E}_{n}^{(r)}\setminus\mathbb{E}^{\text{in},(r)}, & \mathbb{E}_{n}^{\text{in},(r)} & :=\mathbb{E}_{n}^{(r)}\cap\mathbb{E}^{\text{in},(r)}, & \mathbb{E}_{\mathcal{O}} & :=\mathbb{E}\setminus\left(\cup_{r\in\mathbb{N}}\mathbb{E}^{(r)}\right).
\end{align*}
One should also observe that $\mathbb{E}$ is a disjoint union of
the sets defined above:
\begin{align*}
\mathbb{E} & =\mathbb{E}_{\mathcal{O}}\cup\bigcup_{r\in\mathbb{N}}\mathbb{E}^{(r)}\\
 & =\mathbb{E}_{\mathcal{O}}\cup\bigcup_{r\in\mathbb{N}}\bigcup_{n\in\mathbb{Z}}\mathbb{E}_{n}^{(r)}\\
 & =\mathbb{E}_{\mathcal{O}}\cup\bigcup_{r\in\mathbb{N}}\bigcup_{n\in\mathbb{Z}}\left(\mathbb{E}_{n}^{\partial,(r)}\cup\mathbb{E}_{n}^{\text{in},(r)}\right).
\end{align*}

Thus, letting 
\begin{align*}
\Omega_{\mathcal{O}} & =\{0,1\}^{\mathbb{E}_{\mathcal{O}}}, & \Omega_{n}^{(r)} & =\{0,1\}^{\mathbb{E}_{n}^{(r)}}, & \Omega_{n}^{\partial,(r)} & =\{0,1\}^{\mathbb{E}_{n}^{\partial,(r)}}, & \Omega_{n}^{\text{in},(r)} & =\{0,1\}^{\mathbb{E}_{n}^{\text{in},(r)}},
\end{align*}
 we can write
\begin{align*}
\Omega & =\Omega_{\mathcal{O}}\times\prod_{r\in\mathbb{N}}\prod_{n\in\mathbb{Z}}\Omega_{n}^{(r)}\\
 & =\Omega_{\mathcal{O}}\times\prod_{r\in\mathbb{N}}\prod_{n\in\mathbb{Z}}\left(\Omega_{n}^{\partial,(r)}\times\Omega_{n}^{\text{in},(r)}\right).
\end{align*}

Denote $\partial\mathbb{V}_{n}^{(r)}$ to indicate the vertex boundary
of $\mathbb{V}_{n}^{(r)},$ that is,
\begin{align*}
\partial\mathbb{V}_{n}^{(r)} & :=\left\{ (v,m)\in\mathbb{V}_{n}^{(r)}; \dist Gv{U^{(r)}}=1\right\} \\
& \;\cup\left( U^{(r)}\times\{(2L_{r}+2)n\}\right) \cup\left(U^{(r)}\times\{(2L_{r}+2)(n+1)\}\right) .
\end{align*}

Finally, for $A\subset\partial\mathbb{V}_{n}^{(r)}$ and $\omega_{n}^{(r)}\in\Omega_{n}^{(r)}$,
define
\begin{align*}
C_{n}^{(r)}\left(A,\omega_{n}^{(r)}\right)  :=\left\{ (v,m)\in\partial\mathbb{V}_{n}^{(r)}; \exists (v_{0},n_{0})\in A,(v,m)\stackrel{\mathbb{V}_{n}^{(r)}}{\leftrightarrow}(v_{0},n_{0})\right\} .
\end{align*}

Given any $A\subset\mathbb{E}$, let $\mathbb{P}_{.,.}\restriction_{A}$ the measure $\mathbb{P}_{.,.}$ restricted to the sample space $\{0,1\}^A$. With these definitions in hand, we are ready to establish the facts
necessary for the proof of Proposition \ref{prop:worse-q-better-p}.
\begin{lemma}
\label{claim:coupling-world-outside}Let $p,q\in(0,1)$, $\lambda=\min(p,1-p)$.
For any $\varepsilon\in(0,\lambda)$ and $\delta\in(0,1)$ such that
$(q-\delta,q+\delta)\subset [0,1]$, there exists a coupling $\mu_{\mathcal{O}}=(\omega_{\mathcal{O}},\omega'_{\mathcal{O}})$
on $\Omega_{\mathcal{O}}^{2}$ such that
\begin{itemize}
\item $\omega_{\mathcal{O}}  \overset{(d)}{=}\mathbb{P}_{q+\delta,p-\varepsilon}\restriction_{\mathbb{E}_{\mathcal{O}}}$; 
\item $\omega'_{\mathcal{O}}  \overset{(d)}{=}\mathbb{P}_{q-\delta,p+\varepsilon}\restriction_{\mathbb{E}_{\mathcal{O}}}$;
\item $\omega_{\mathcal{O}}  \leq\omega'_{\mathcal{O}}\quad\text{a.s.}$.
\end{itemize}

\end{lemma}
\begin{proof}
This construction is standard. Let $Z=(Z_{1},Z_{2})\in\Omega_{\mathcal{O}}^{2}$
be a pair of random elements defined in some probability space, such
that the marginals $Z_{1}$ and $Z_{2}$ are independent on every
edge of $\mathbb{E}_{\mathcal{O}}$ and assign each edge to be open
with probabilities $p-\varepsilon$ and $\frac{2\varepsilon}{1-p+\varepsilon}$,
respectively. Taking $\omega_{\mathcal{O}}=Z_{1}$ and $\omega'_{\mathcal{O}}=Z_{1}\lor Z_{2}$,
define $\mu_{\mathcal{O}}$ to be the distribution of $(\omega_{\mathcal{O}},\omega'_{\mathcal{O}})$
and the claim readily follows.
\end{proof}

The next lemma is one of the fundamental facts established in \cite{SV},
so we refer the reader to the paper for a proof of the statement.
\begin{lemma}
\label{claim:coupling-inside-cylinders}Let $p,q\in(0,1)$, $\lambda=\min(p,1-p)$,
$r\in\mathbb{N}$. For any $\varepsilon\in(0,\lambda)$, there exists
an $\eta^{(r)}>0$ such that if $\delta\in\left(0,\eta^{(r)}\right)$,
there is a coupling $\mu_{n}^{(r)}=(\omega_{n}^{(r)},{\omega'}_{n}^{(r)})$ on $\Omega_{n}^{(r)}\times\Omega_{n}^{(r)}$
with the following properties:

\begin{itemize}
\item $\omega_{n}^{(r)}  \overset{(d)}{=}\mathbb{P}_{q+\delta,p-\varepsilon}\restriction_{\mathbb{E}_{n}^{(r)}}$; 
\item ${\omega'}_{n}^{(r)} \overset{(d)}{=}\mathbb{P}_{q-\delta,p+\varepsilon}\restriction_{\mathbb{E}_{n}^{(r)}}$;
\item $C_{n}^{(r)}\left(A,\omega_{n}^{(r)}\right)\subset C_{n}^{(r)}\left(A,{\omega'}_{n}^{(r)}\right)$
for every $A\in\partial\mathbb{V}_{n}^{(r)}$ almost surely.
\end{itemize}

Moreover, the value $\eta^{(r)}>0$ depends only on the choice of $q$, $p$,
$\varepsilon$ and the graph $\left(\mathbb{V}_{0}^{(r)},\mathbb{E}_{0}^{(r)}\right)$.
\end{lemma}
The last ingredient used in the proof Proposition \ref{prop:worse-q-better-p}
is the following fact:
\begin{lemma}
\label{claim:uniform-bound}If $\sup_{r\in\mathbb{N}}\left|U^{(r)}\right|<\infty$
then for any $\epsilon>0$ fixed, the sequence $\left\{ \eta^{(r)}\right\} _{r\in\mathbb{N}}$
in Lemma \ref{claim:coupling-inside-cylinders} may be chosen bounded away from
$0$.
\end{lemma}
\begin{proof}
From Lemma \ref{claim:coupling-inside-cylinders}, it follows that,
for every $r\in\mathbb{N}$, the value $\eta^{(r)}>0$ depends on
the choice of $q$, $p$, $\varepsilon$ and the graph $\left(\mathbb{V}_{0}^{(r)},\mathbb{E}_{0}^{(r)}\right)$.
Note that while the values of $q$, $p$ and $\varepsilon$ are the
same for different values of $r\in\mathbb{N}$, the graphs $\left(\mathbb{V}_{0}^{(r)},\mathbb{E}_{0}^{(r)}\right)$
may differ. However, there are only a finite number of possible graphs
for $\left(\mathbb{V}_{0}^{(r)},\mathbb{E}_{0}^{(r)}\right)$ to assume.
As a matter of fact, the graph $\left(\mathbb{V}_{0}^{(r)},\mathbb{E}_{0}^{(r)}\right)$
is obtained from the vertex set $U^{(r)}\cup\partial U^{(r)}$ and
from the edges with both endpoints in $U^{(r)}\cup\partial U^{(r)}$.
Since $\sup_{r\in\mathbb{N}}\left|U^{(r)}\right|<\infty$ and
$G$ is of limited degree, we have that $M:=\sup_{r\in\mathbb{N}}\left|U^{(r)}\cup\partial U^{(r)}\right|<\infty$.
Since there are only a finite number of graphs of limited degree with
at most $M$ vertices, the claim regarding $\left(\mathbb{V}_{0}^{(r)},\mathbb{E}_{0}^{(r)}\right)$
follows, that is, $\eta:=\inf_{r\in\mathbb{N}}\eta^{(r)}>0$.
\end{proof}

\begin{proof}[Proof of Proposition \ref{prop:worse-q-better-p}]
From Lemmas \ref{claim:coupling-inside-cylinders} and \ref{claim:uniform-bound}
we have the following result:
Let $p,q\in(0,1)$, $\lambda=\min(p,1-p)$. For any $\varepsilon\in(0,\lambda)$,
there exists an $\eta>0$ such that if $\delta\in(0,\eta)$, there
is a family of couplings $\left\{ \mu_{n}^{(r)}\right\} _{\substack{r\in\mathbb{N}\\
n\in\mathbb{Z}
}
}$, with each $\mu_{n}^{(r)}=(\omega_{n}^{(r)},{\omega'}_{n}^{(r)})$ defined on $\Omega_{n}^{(r)}\times\Omega_{n}^{(r)}$
and having the following property:

\begin{itemize}
\item $\omega_{n}^{(r)}  \overset{(d)}{=}\mathbb{P}_{q+\delta,p-\varepsilon}\restriction_{\mathbb{E}_{n}^{(r)}}$; 
\item ${\omega'}_{n}^{(r)}  \overset{(d)}{=}\mathbb{P}_{q-\delta,p+\varepsilon}\restriction_{\mathbb{E}_{n}^{(r)}}$;
\item $C_{n}^{(r)}\left(A,\omega_{n}^{(r)}\right)\subset C_{n}^{(r)}\left(A,{\omega'}_{n}^{(r)}\right)$
for every $A\in\partial\mathbb{V}_{n}^{(r)}$ almost surely.
\end{itemize}

Let $\mu_{\mathcal{O}}$ be the coupling of Lemma \ref{claim:coupling-world-outside}.
Defining the coupling measure $\mu$ on $\Omega^{2}$ by
\[
\mu=\mu_{\mathcal{O}}\times\prod_{r\in\mathbb{N}}\prod_{n\in\mathbb{Z}}\mu_{n}^{(r)},
\]
it is clear that if $(\omega,\omega')\sim\mu$, then $\omega\overset{(d)}{=}\mathbb{P}_{q+\delta,p-\varepsilon}$,
$\omega'\overset{(d)}{=}\mathbb{P}_{q-\delta,p+\varepsilon}$, and
almost surely $(v,0)\leftrightarrow\infty$ in $\omega$ implies $(v,0)\leftrightarrow\infty$ in ${\omega'}$,
for every $v\in V\setminus\left(\cup_{r\in\mathbb{N}}U^{(r)}\right)$.
\end{proof}

\section*{Acknowledgements} BNBL is partially suported by CNPq. Both authors would like to thank CAPES for the financial support.

\end{document}